\documentclass[12pt]{amsart}
\usepackage{amsmath,amsthm,amssymb,amsfonts,amscd}
\usepackage{amsxtra}
\usepackage{mathrsfs}
\usepackage{color, graphics}
\usepackage[all]{xy}
\usepackage{hyperref}
\usepackage{pstricks,pst-plot}
\usepackage{verbatim}
\usepackage{enumerate}
\def\pp{{\mathfrak{p}}}\def\mm{{\mathfrak{m}}}
\def\kk{{\Bbbk}}
\def\qq{{\mathfrak{q}}}
  \def\m{{\mathfrak{m}}} \def\Z{{\mathbb{Z}}}\def\N{{\mathbb{N}}}
\def\O{{\mathcal{O}}}
\def\ZZ{{\mathbb{Z}}}
 \def\R{{\mathbb{R}}} \def\C{{\mathbb{C}}}
\def\Hom{{\mathrm{Hom}}}  
  
\def\Spec{{\mathrm{Spec\; }}}
\def\Proj{{\mathrm{Proj\; }}}

\def\Ext{\operatorname{Ext}}

\def\height{\operatorname{ht}}

\def\CC1{{\mathsf{CC}_1(k)}}

\def\Pic{\operatorname{Pic}}

\DeclareMathOperator{\reg}{reg}
\DeclareMathOperator{\depth}{depth}
\DeclareMathOperator{\chara}{char}
\DeclareMathOperator{\Supp}{Supp}

\theoremstyle{plain}
\newtheorem{theorem}{Theorem}[section]
\newtheorem{corollary}[theorem]{Corollary}
\newtheorem{proposition}[theorem]{Proposition}

\newtheorem{definition}[theorem]{Definition}

\newtheorem{lemma}[theorem]{Lemma}
\newtheorem{question}[theorem]{Question}
\newtheorem{eg}[theorem]{Example}
\newtheorem{remark}[theorem]{Remark}

\newtheorem{setup}[theorem]{Set up}

\newtheorem*{theorem*}{Theorem}

\newcommand \tv[1]{\textcolor{violet}{#1}}

\begin{document}

\title{Regularity, singularities and $h$-vector of graded algebras}
\date{\today}

\author[H. Dao]{Hailong Dao}
\address{Department of Mathematics, University of Kansas,
Lawrence, KS 66045-7523, USA} \email{hdao@ku.edu}

\author[L. Ma] {Linquan Ma}
\address{Department of Mathematics, Purdue University, West Lafayette, IN 47906, USA}
\email{ma326@purdue.edu}

\author[M. Varbaro]{Matteo Varbaro}
\address{Dipartimento di Matematica,  Universit\`a di Genova
Via Dodecaneso, 35
16146 Genova, Italy}
\email{varbaro@dima.unige.it}


\baselineskip 16pt \footskip = 32pt

\begin{abstract}
Let $R$ be a standard graded algebra over a field. We investigate how the singularities of $\Spec R$ or $\Proj R$ affect the $h$-vector of $R$, which is the coefficients of the numerator of its  Hilbert series. The most concrete consequence of our work asserts that  if $R$ satisfies Serre's condition $(S_r)$ and have reasonable singularities (Du Bois on the punctured spectrum or $F$-pure), then $h_0,\dots, h_r\geq 0$. Furthermore the multiplicity of $R$ is at least $h_0+h_1+\dots +h_{r-1}$.  We also prove that equality in many cases forces $R$ to be Cohen-Macaulay.  The main technical tools are sharp bounds  on regularity of certain $\Ext$ modules, which can be viewed as Kodaira-type vanishing statements for Du Bois and $F$-pure singularities. Many corollaries are deduced, for instance that nice singularities of small codimension  must be Cohen-Macaulay. Our results build on and extend previous work by de Fernex-Ein,  Eisenbud-Goto, Huneke-Smith, Murai-Terai  and others.
\end{abstract}
\thanks{}

\maketitle

\section{Introduction}
Let  $R=\kk[x_1,\dots, x_n]/I$ be a standard graded algebra over a field $\kk$. Suppose $R$ is a domain. How many independent quadrics can $I$ contain? Such a question has long been studied in algebraic geometry, and some beautiful answers are known, at least to experts. For example, if $\kk$ is algebraically closed, and $I$ is prime and contains no linear forms, then the number of quadratic minimal generators of $I$ is at most  $\binom{n-d+1}{2}$ where $d=\dim R$. Furthermore, equality happens if and only if $R$ is a variety of minimal degree (see Theorem \ref{h2} for a generalization of this result to non-reduced schemes).

In another vein, a striking result by Bertram-Ein-Lazarsfeld (\cite[Corollary 2]{BEL}) states that if $\Proj R $ is smooth and the sum of $n-d$ highest degrees among the minimal generators of $I$ is less than $n$, then $R$ is Cohen-Macaulay. We will also extend this result to mild singularities, see below.

The common theme in the aforementioned statements is how the singularities of $R$ affect its Hilbert function. This article grew out of an attempt to better understand this phenomenon. It turns out that such numerical information is most conveniently expressed via the $h$-vector. Recall that we can write the Hilbert series of $R$
 \[H_R(t)= \sum_{i\in\N} (\dim_{\kk}R_i)t^i = \frac{h_0+h_1t+ \cdots + h_st^s}{(1-t)^d}. \]
We will call the vector of coefficients $(h_0,h_1,\ldots ,h_s)$ the $h$-vector of $R$. The bound on the number of quadrics mentioned above can be easily seen to be equivalent to the statement that $h_2\geq 0$. Furthermore, it is well-known and easy to prove that if $R$ is Cohen-Macaulay then $h_i\geq 0$ for all $i$.  A very concrete motivation for our work comes from the following question:

\begin{question}\label{mainques}
Suppose $R$ satisfies Serre's condition $(S_r)$. When is it true that $h_i \geq 0$ for $i=0,\ldots , r$?
\end{question}

An answer for this question when $I$ is a square-free monomial ideal is contained in a beautiful result of  Murai-Terai  (the $h$-vector of a simplicial complex $\Delta$ equals the $h$-vector of its Stanley-Reisner ring $\kk[\Delta]$).

\begin{theorem*}[\cite{MT}]
Let $\Delta$ be a simplicial complex with $h$-vector $(h_0,h_1,\ldots ,h_s)$. If $\kk[\Delta]$ satisfies $(S_r)$ for some field $\kk$, then $h_i\geq 0$ for $i=0,\ldots , r$.
\end{theorem*}

In general, some additional assumptions are needed as shown in Remark \ref{rem:negh2} where, for any $n\geq 3$, it is constructed a standard graded $\kk$-algebra $R$ satisfying $(S_{n-1})$ and having $h$-vector $(1,n,-1)$. It even turns out that such an $R$ has  Castelnuovo-Mumford regularity 1 and is Buchsbaum. Contrary to this, we can surprisingly show the following, which is our main technical result:

\begin{theorem}\label{main0}
Let $R$ be a standard graded algebra over a field and $(h_0,\ldots ,h_s)$ the  $h$-vector of $R$. Assume $R$ satisfies $(S_r)$ and either:
\begin{enumerate}

\item $\chara \kk=0$ and $X=\Proj R$ is Du Bois.
\item $\chara \kk=p>0$ and $R$ is $F$-pure.
\end{enumerate}
Then $h_i\geq 0$ for $i=0,\ldots , r$. Also, $h_r+ h_{r+1}+\dots+h_s \geq 0$, or equivalently $R$ has multiplicity at least $h_0+h_1+\dots +h_{r-1}$.

If furthermore $R$ has Castelnuovo-Mumford regularity less than $r$ or $h_i=0$ for some $i\leq r$,  then $R$ is Cohen-Macaulay.
\end{theorem}

Note that Theorem \ref{main0} generalizes the work of Murai-Terai, since a Stanley-Reisner ring is Du Bois in characteristic $0$ and $F$-pure in characteristic $p$. On the other hand the class of varieties satisfying the assumptions of Theorem \ref{main0} are much more than the ones given by Stanley-Reisner rings: e.g. varieties $X=\Proj R$ with only simple normal crossing singularities with $\depth R\geq r$ in characteristic 0, or smooth globally $F$-split varieties $X=\Proj R$ with $\depth R\geq r$ in positive characteristic.
The key point is to establish good bounds on certain $\Ext$ modules, just like what Murai and Terai did, but for ideals that are not necessarily monomial and square-free. We remark that, for square-free monomial ideals, the inequality $h_r+ h_{r+1}+\dots+h_s \geq 0$ has later been improved in \cite{GPSY}.

One can deduce other surprising consequences from our results. We discuss one of them now.  There is a well-known but mysterious theme in algebraic geometry that ``nice singularities of small codimension" should be Cohen-Macaulay. Perhaps the most famous example is Hartshorne's conjecture that a smooth subvariety in $\mathbb P_{\kk}^n$ of codimension $e$ has to be a complete intersection provided $e< n/3$. Although this conjecture is wide-open, numerous partial results has been obtained, and many of them assert that such variety has to be projectively Cohen-Macaulay.  One or our corollaries can be viewed as another confirmation of this viewpoint (extending \cite[Corollary 1.3]{dFE}).

\begin{corollary}[Corollary \ref{cor2}]
Let $R=\kk[x_1,\dots, x_n]/I$ be a standard graded algebra over a field $\kk$ of characteristic $0$ with $e=\height I$ and $d=\dim R$. Let $d_1 \geq d_2 \geq \dots $ be the degree sequence of a minimal set of generators for $I$. Assume that $R$ is unmixed, equidimensional, Cohen-Macaulay in codimension $l$ and $X=\Proj R$ has only MJ-log canonical singularities (e.g., only Du Bois singularities). If $e+l \geq d_1+\dots +d_e$, then $R$ is Cohen-Macaulay.
\end{corollary}

For instance, the corollary asserts that if $I$ is generated by quadrics and $\Proj R$ has mild singularities, then $R$ is Cohen-Macaulay provided that it is Cohen-Macaulay in codimension $e$. Under stronger assumptions, it follows that $R$ is Gorenstein or even complete intersection, see Corollary \ref{cor3}.

We now briefly describe the content and organization of the paper. Section \ref{prelim} establishes all the important preparatory results. It contains a careful analysis of a property which we call $(MT_r)$. This is a condition on regularity of certain $\Ext$ modules studied by Murai-Terai (these modules are sometimes called deficiency modules in the literature, see \cite{Sc}). It turns out that this regularity condition is precisely what we need to prove all the statements of the main theorem. We give a self-contained and independent proof that $(MT_r)$ implies non-negativity of $h_i$ for $i\leq r$, when equalities occur,  and also that $(MT_r)$ is preserved under generic hyperplane section. Our proofs work also for modules.

What remains, which is our main technical contribution, is to establish the condition $(MT_r)$ for algebras with nice singularities. We achieve this goal in Sections \ref{main} (characteristic $0$) and \ref{mainp} (positive characteristic). Our results there give strong bounds on regularity of the Matlis dual of the local cohomology modules of such algebras, and are perhaps of independent interests. As mentioned above, they can be viewed as Kodaira-vanishing type statements.

In Section \ref{apps} we put everything together to prove Theorem \ref{main0} and give several applications. Our bounds imply stringent conditions on Hilbert functions of algebras satisfying assumptions of Theorem \ref{main0}, see Corollary \ref{cor1}. One can also show that certain algebras with nice singularities and small codimension must be Cohen-Macaulay, see Corollary \ref{cor2}. We also write down perhaps the strongest possible statement for $h_2(R)$, which can be viewed as generalization of the classical bound $e(R) \geq 1+ e$ where $e$ is the embedding codimension, and the equality case which forces $R$ to be rings of minimal multiplicity.

In the last section, we propose an open question which can be viewed as a very strong generalization of our result and prove a first step toward such problem, which generalized one of the main results of Huneke and Smith in \cite{HuSm}.

\subsection*{Acknowledgement} The first author is partially supported by Simons Foundation Collaboration Grant 527316. The second author is partially supported by NSF Grant DMS \#2302430, NSF FRG Grant \#1952366, and a fellowship from the Sloan Foundation, and was partially supported by NSF Grant DMS \#1836867/1600198 and \#1901672 when working on this article. The third author is supported by MIUR Grant PRIN \#2020355B8Y. The second author would like to thank Karl Schwede for many discussions on Du Bois singularities. The first and third authors would like to thank Kangjin Han for useful discussions on $h$-vectors and singularities, which started this project. We also thank Tommaso de Fernex and S\'andor  Kov\'acs for some helpful email exchanges. Finally, we thank the referee for very detailed comments that lead to improvement of this paper.

\section{Preliminaries}\label{prelim}

\begin{setup}\label{setup}
We will use the following notations and conventions throughout the paper:

\begin{enumerate}[(i)]
\item $n$ is a positive integer;
\item  $\kk$ is a field;
\item $S$ is the standard graded polynomial ring in $n$ variables over $\kk$;
\item $\mm$ is the irrelevant ideal of $S$;
\item $\omega_S\cong S(-n)$ is the canonical module of $S$;
\item $I\subset S$ is a homogeneous ideal of height $e$;
\item  $R=S/I$ is the quotient of dimension $d=n-e$;
\item  $X=\Proj R$ is the projective scheme associated to $R$ ($\dim X=d-1$).
\item We say that $R$ ($X=\Proj R$) has some property in codimension $t$ if $R_{\pp}$ has that property whenever $\pp\in\Spec R$ ($\pp\in X$) has height at most $t$.
\item We say that $R$ ($X=\Proj R$) satisfies Serre's condition $(S_i)$ if $\depth R_{\pp}\geq \min\{i, \height \pp\}$ for all $\pp\in \Spec R$ (for all $\pp\in X$).
\item Given a graded $S$-module $M$ by $M_{\leq r}$ we mean its submodule generated by the elements in $M$ of degree $\leq r$.
\item For a nonzero finitely generated graded $S$-module $M$ of dimension $d$, let $H_M(t)$ denote the Hilbert series of $M$, and $(h_a(M),h_{a+1}(M),\ldots ,h_s(M))$ denote the $h$-vector of $M$, given by:
\[H_M(t)= \sum_{i\in\Z} (\dim_{\kk}M_i)t^i = \frac{h_at^a+h_{a+1}t^{a+1}+ \cdots + h_st^s}{(1-t)^d},  \]
We also write $p_M(t)$ for the numerator and $c_r(M)$ for the quantity $h_r(M) +h_{r+1}(M)+\dots + h_s(M)$. Note that  $c_r(M)= e(M)-\sum_{i<r} h_i(M)$,where $e(M)$ denotes the Hilbert-Samuel multiplicity.
\item Given a finitely generated $S$-module $M$ its Castelnuovo-Mumford regularity is $\reg(M)=\sup\{i+j:H^i_{\mm}(M)_j\neq 0\}$. In particular the regularity of the zero module is $-\infty$.
\end{enumerate}

\end{setup}

Let us recall the following result of Murai and Terai (\cite[Theorem 1.4]{MT}):
\begin{theorem}\label{t:MT}
If $\reg(\Ext_S^{n-i}(R,\omega_S))\leq i-r$ $ \forall \ i=0,\ldots ,d-1$, then $h_i(R)\geq 0$ for $i\leq r$.
\end{theorem}

Motivated by this result, we make the following definition.
\begin{definition}\label{defMT}
Consider a finitely generated graded module $M$ over $S$ of dimension $d$. We say that $M$  satisfies condition $(MT_r)$ if  $\reg(\Ext_S^{n-i}(M,\omega_S))\leq i-r$ $ \forall \ i=0,\ldots ,d-1$.
\end{definition}

The analysis of this notion led us to a more direct proof of Theorem \ref{t:MT}. We first note a few preparatory results.
The following facts are trivial consequences of graded duality:
\begin{proposition}\label{easy}
\begin{enumerate}
\item  The condition $(MT_r)$ depends only on $R$ and does not depend on the presentation $R=S/I$ (we won't need this).
\item If $M$ is Cohen-Macaulay, then it is $(MT_r)$ for any $r$.
\item If $N = H^i_{\m}(M)$ has finite length, then $\reg(\Ext_S^{n-i}(M,\omega_S))\leq i-r$ if and only if $N_{<r-i} =0$.
 \end{enumerate}

\end{proposition}

\begin{proposition}\label{linear}
Suppose $\kk$ is infinite. Let $M$ be a finitely generated graded $S$-module. For a generic linear form $l$:
$$ \reg(M) \geq \max\{\reg{0:_Ml}, \reg M/lM\}$$
\end{proposition}
\begin{proof}
This is well-known and easy, but we give a proof for convenience. Let $N=H_\m^0(M)$. Since $\kk$ is infinite, for a general linear form $l$, we have $l\notin P$ for every $P\in \text{Ass}(M)\backslash\{\m\}$. In particular, we know that $0:_Ml$ has finite length and thus $0:_Ml \subseteq N$, so it's regularity is at most $\reg(N)\leq \reg(M)$.
Consider the short exact sequence $0\to N\to M\to M'\to 0$. We can assume $l$ is regular on $M'$, which leads to $0\to N/lN\to M/lM\to M'/lM'\to 0$. Since $\reg(N/lN)\leq \reg(N)$ and $\reg(M'/lM') = \reg(M') \leq \reg (M)$, the result follows.
\end{proof}

\begin{proposition}\label{mtr1}
Suppose $\kk$ is infinite. Let $M$ be a finitely generated graded $S$-module. Assume that  $M$ is $(MT_r)$, then so are the following modules:
$M'= M/H^0_{\m}(M), M'/lM', M/lM$,  for a generic linear form $l$.
\end{proposition}

\begin{proof}
Let $K_i(M)$ denote the module $\Ext_S^{n-i}(M,\omega_S)$. The short exact sequence $0 \to H_{\m}^0(M) \to M \to M' \to 0$ tells us that $K_0(M')=0$ and $K_i(M') = K_i(M)$ for $i\geq 1$. So $M'$ is $(MT_r)$.
For the rest, we assume $d=\dim M\geq 1$. As $l$ is generic, it is regular on $M'$, so from the short exact sequence $0 \to M'(-1) \to M' \to M'/lM' \to 0$ we obtained for each $i\leq d-1=\dim M'/lM'$:
$$0 \to K_{i+1}(M'(-1))/l K_{i+1}(M'(-1)) \to K_i(M'/lM') \to 0:_{K_i(M')}l \to 0    $$
The bound on regularity of $K_i(M'/lM')$ follows from the property $(MT_r)$ for $M'$ and Proposition \ref{linear}.
For the module $M/lM$, note that we have the short exact sequence $0\to N/lN\to M/lM\to M'/lM'\to 0$ where $N=H_\m^0(M)$. Another application of the long exact sequence of $\Ext$ and the fact that $M'/lM'$ is $(MT_r)$ proves what we need.
\end{proof}

The next proposition tracks the behavior of $h$-vector  modulo a generic linear form for a module satisfying $(MT_r)$. Let $c_r(M)$ be the quantity: $h_r(M)+ h_{r+1}(M)+\cdots = e(M)-\sum_{i<r}h_i(M)$.
\begin{proposition}\label{mtr2}
Suppose $\kk$ is infinite. Let $M$ be a finitely generated graded $S$-module of dimension $d\geq 1$. Let $l$ be a generic linear form. Let $N=H_\m^0(M)$ and $M'=M/N$.  Assume that $N_{<r}=0$ (for example, if $M$ is $(MT_r)$). Then:
\begin{enumerate}
\item $h_i(M) =h_i(M/lM)$ for $i<r$.
\item $h_r(M) \geq h_r(M/lM)$.
\item $c_r(M) =c_r(M'/lM')$.
\item $c_r(M) =   c_r(M/lM)$ if $d>1$.
\end{enumerate}
\end{proposition}

\begin{proof}
First of all $M$ is $(MT_r)$ implies $N_{<r}=0$, see Proposition \ref{easy}. Since $p_M(t) = p_{N}(t)(1-t)^d + p_{M'}(t)$, we have $h_i(M) = h_i(M')$ for $i<r$ and $h_r(M) = h_r(M')+h_r(N)$, as well as $c_r(M)=c_r(M')$. As $l$ is generic, it is regular on $M'$ and so $h_i(M'/lM') = h_i(M')$ for all $i$. We also have the short exact sequence $0\to N/lN\to M/lM\to M'/lM'\to 0$, which gives:
\begin{itemize}
\item $h_i(M/lM) = h_i(M'/lM') = h_i(M') = h_i(M)$ for $i<r$.
\item $h_r(M/lM) =h_r(M'/lM') +h_r(N/lN) \leq h_r(M)$.
\item $c_r(M)=c_r(M') = c_r(M'/lM')$.
\item $c_r(M/lM) = c_r(M'/lM') = c_r(M)$ if $d>1$. \qedhere
\end{itemize}
\end{proof}

We can now give a direct and simple proof of Theorem \ref{t:MT}, extended to the module case.

\begin{theorem}\label{mtr3}
Let $M$ be a module satisfying $(MT_r)$. Then $h_i(M)\geq 0$ for $i\leq r$ and also $c_r(M)\geq 0$.
\end{theorem}

\begin{proof}
We can extend $\kk$ if necessary and hence assume it is infinite. Since $M$ is $(MT_r)$, we have $N_{<r}=0$ where $N=H_\m^0(M)$, see Proposition \ref{easy}. By Propositions \ref{mtr1} and \ref{mtr2}, we can cut down by a generic linear system of parameters of length $d$ without increasing any of the relevant invariants (for the statement about $c_r(M)$, at the step $d=1$ we need to use $M'/lM'$ and not $M/lM$). But for Artinian modules all inequalities are trivial.\end{proof}

\begin{proposition}\label{mtr4}
If $M$ is $(MT_r)$ and either:
\begin{enumerate}
\item $\reg(M)<r$, or
\item $h_i(M)=0$ for some $i\leq r$ and $M$ is generated in degree $0$.
\end{enumerate}
Then $M$ is Cohen-Macaulay.
\end{proposition}

\begin{proof}
We can extend $\kk$ if necessary and hence assume that it is infinite. We use induction on $d=\dim M$ and therefore assume $d\geq 1$.
Assume first that $\reg(M)<r$.  Let $N =H_\m^0(M)$, then $N_{<r}=0$. But $\reg(N)\leq \reg(M) <r$, so $N=0$. Thus $\depth M\geq 1$.
By Propositions \ref{linear} and \ref{mtr1} as well as induction hypotheses, $M/lM$ is Cohen-Macaulay for a generic $l$, which implies that $M$ is Cohen-Macaulay.

Now assume that $h_i(M)=0$ for some $i\leq r$. Again, we shall show that $N= H_\m^0(M)=0$. Let  $M'=M/N$. Then as in the proof of Proposition \ref{mtr1} and \ref{mtr2}, $M'$ is still $(MT_r)$ and $h_i(M') = h_i(M)=0$ (when $i=r$, we have $h_r(M')=h_r(M) - h_r(N) \leq h_r(M)=0$, but since $M'$ is $(MT_r)$, we also know that $h_r(M')\geq 0$ by Theorem~\ref{mtr3}, so $h_r(M')=0$). Passing to $M'/lM'$ preserves everything, so by induction $M'/lM'$ is Cohen-Macaulay, and so is $M'$. It follows that $h_j(M')=0$ for $j\geq i$ and $\reg M'\leq i-1$.

Write $M= F/P$ where $P$ is the first syzygy of $M$. Then $N=Q/P$ for $Q\subseteq F$, and $M'=F/Q$. Since $\reg(M')\leq i-1$, $Q$ is generated in degree at most $i$. But $N_{\leq i}=0$ (we already know that $N_{<r}=0$, if $i=r$ then the equality $h_r(M)= h_r(N)+h_r(M')$ forces $h_r(N)=N_r=0$). Thus $Q= Q_{\leq i} \subseteq P$, and $N=0$. The rest follows as in the first part.
\end{proof}

We next list a few examples showing the sharpness of Theorem~\ref{mtr3} and Proposition~\ref{mtr4}.

\begin{eg}
\begin{enumerate}
    \item Let $R= \kk[a^4,a^3b,ab^3,b^4]$. Then $R$ is $(MT_2)$ but not $(MT_3)$. The $h$-vector of $R$ is $(1,2,2,-1)$. In particular, $h_3(R)<0$, which shows the sharpness of Theorem~\ref{mtr3}.
  \item Let $R=\kk[x,y]/(x^3,x^2y)$. Then $R$ is $(MT_2)$ and $c_2(R)=0$ but $R$ is not Cohen-Macaulay. So the analog of Proposition~\ref{mtr4} with $c_r$ in place of $h_r$ is false. In this example, the $h$-vector is $(1,1,1,-1)$.
  \item One cannot drop the condition that $M$ is $(MT_r)$ in Proposition~\ref{mtr4}. Let $R=\kk[x,y,u,v]/(xuv,yu)$. Then the $h$-vector of $R$ is $(1,1,0,-1)$. In particular, $h_2(R)=0$ but $R$ is not Cohen-Macaulay. Of course, $R$ fails $(MT_2)$.
  \item One cannot drop the condition that $M$ is generated in degree $0$ in Proposition~\ref{mtr4}. Let $S=\kk[x]$, $M=S\oplus \kk(-3)$. Then $M$ is $(MT_2)$, $h_2(M)=0$, but $M$ is not Cohen-Macaulay.
\end{enumerate}
\end{eg}


The following connects Serre's condition $(S_i)$ to dimension of $\Ext$, it appears without proof in \cite[Proposition 3.51]{Vas}:

\begin{proposition}\label{algSl}
If $R$ is equidimensional, it satisfies $(S_r)$ if and only if
\[\dim \Ext_S^{n-i}(R,S) \leq i-r \ \ \ \forall \ i<\dim R\]
(the Krull dimension of the zero-module is $-\infty$).
\end{proposition}
\begin{proof}
Recall that $\dim R=d$ and $\height I=e$. Assume first that $R$ satisfies $(S_r)$. If $i<r\leq \depth R$, there is nothing to prove because $\Ext_S^{n-i}(R,S)=0$. Otherwise, let $\mathfrak{p}$ be a prime ideal of $S$ containing $I$ of height $h$. Notice that, because $R$ is equidimensional, $\dim R_{\pp}=h-e$. If $h<n-i+r$, then
\[\Ext_S^{n-i}(R,S)_{\mathfrak{p}}\cong \Ext_{S_{\mathfrak{p}}}^{n-i}(R_{\mathfrak{p}},S_{\mathfrak{p}})=0,\]
because $n-i=h-(h-n+i)$ and $\depth R_{\mathfrak{p}}>h-n+i$. To see it, note that $h-n+i<n-i+r-n+i=r$ and $\dim R_{\mathfrak{p}}>h-n+i$, so because $R$ satisfies $(S_r)$
\[\depth R_{\mathfrak{p}}=\min\{r,\dim R_{\mathfrak{p}}\}>h-n+i.\]
So $\mathfrak{p}\notin \mathrm{Supp} \Ext_S^{n-i}(R,S)$ whenever $\height \mathfrak{p}<n-(i-r)$, i.e. $\dim \Ext_S^{n-i}(R,S)\leq i-r$.
%
%

For the other direction, for all $h=e,\ldots ,n$ let us denote by $V_h$ the set of prime ideals of $S$ of height $h$ containing $I$. We have the following chain of equivalences:
\begin{eqnarray*}
\mbox{$R$ satisfies $(S_r)$ } & \iff \\
\depth R_{\mathfrak{p}}=\min\{r,\dim R_{\pp}\}=:b \ \ \ \forall \ h=e,\ldots ,n, \ \forall \ \pp \in V_h & \iff \\
\depth R_{\mathfrak{p}}=\min\{r,h-e\}=:b \ \ \ \forall \ h=e,\ldots ,n, \ \forall \ \pp \in V_h & \iff \\
H_{\pp S_{\pp}}^i(R_{\pp})=0 \ \ \ \forall \ h=e,\ldots ,n, \ \forall \ \pp \in V_h, \ \forall \ i<b & \iff \\
\Ext_{S_{\pp}}^{h-i}(R_{\pp},S_{\pp})=0 \ \ \ \forall \ h=e,\ldots ,n, \ \forall \ \pp \in V_h, \ \forall \ i<b & \iff \\
\dim \Ext_{S}^{h-i}(R,S)<n-h \ \ \ \forall \ h=e,\ldots ,n, \ \forall \ i<b & \\
\end{eqnarray*}
(where in the equivalence between the second and third line we used the equidimensionality of $R$). Since $i<b\leq h-e$, we have $n-h+i<n-e=d$. If $\dim \Ext_S^{n-i}(R,S) \leq i-r \ \ \ \forall \ i=0,\ldots ,d-1$, then
\[\dim \Ext_{S}^{h-i}(R,S)=\dim \Ext_{S}^{n-(n-h+i)}(R,S)\leq n-h+i-\ell< n-h \]
for all $h=e,\ldots ,n$ (because $i<b\leq r$).
\end{proof}

\begin{remark}
In Proposition \ref{algSl}, the equidimensionality of $R$ is needed only in the case `` $r=1$". Indeed, if $R$ satisfies $(S_r)$ with $r\geq 1$ then it has no embedded prime, and if $R$ satisfies $(S_r)$ with $r\geq 2$ then it is equidimensional by \cite[Remark 2.4.1]{Ha}.

On the other hand, an argument similar to  the one used to prove Proposition \ref{algSl} shows that  $\dim R/\pp=\dim R$ for all associated prime ideals if and only if $\dim \Ext_S^{n-i}(R,S) <i \ \ \ \forall \ i<\dim R$.
\end{remark}

We conclude this section with the following remark that shows that we cannot hope to answering ``always!'' to Question \ref{mainques}.


\begin{remark}\label{rem:negh2}
Let $S=\kk[x_i,y_i: i=1,\ldots ,n]$ and $I\subseteq S$ the following ideal:
\[I=(x_1,\ldots ,x_n)^2+(x_1y_1+\cdots +x_ny_n).\]
Setting $R=S/I$, we claim the following properties:
\begin{itemize}
\item[(i)] $R$ is an $n$-dimensional ring satisfying $(S_{n-1})$;
\item[(ii)] $\reg R=1$;
\item[(iii)] $R$ is not Cohen-Macaulay;
\item[(iv)] $H_R(t)= \displaystyle \frac{1+nt-t^2}{(1-t)^n}$;
\item[(v)] $H_{\mathfrak{m}}^{n-1}(R)=H_{\mathfrak{m}}^{n-1}(R)_{2-n}=\kk$, thus $R$ is Buchsbaum.
\end{itemize}
For it, let us notice that, for any term order on $S$, the set $\{x_ix_j:1\leq i \leq j\leq n\}\cup \{x_1y_1+\ldots +x_ny_n\}$ is a Gr\"obner basis by Buchberger's criterion. So we can choose a term order such that the initial ideal of $I$ is
\[J=(x_1,\cdots ,x_n)^2+(x_1y_1).\]
So $\depth R\geq \depth S/J=n-1$ and $H_R(t)=H_{S/J}(t)=H_{P/J\cap P}(t)/(1-t)^{n-1}$, where $P=\kk[x_1,\ldots ,x_n,y_1]$. Note that $\dim_{\kk} (P/J\cap P)_0=1$, $\dim_{\kk} (P/J\cap P)_1=n+1$ and $\dim_k (P/J\cap P)_d=n$ for any $d>1$. So $H_{P/J\cap P}(t)=(1+nt-t^2)/(1-t)$, yielding
\[H_R(t)= \frac{1+nt-t^2}{(1-t)^n}.\]
In particular, $R$ is not Cohen-Macaulay, therefore $\depth R=n-1$. We have $1\leq \reg R\leq \reg S/J$. Since it is a monomial ideal, it is simple to check that $J$ has linear quotients, so $\reg S/J=1$.
Finally, take a prime ideal $\pp\in\Proj S$ containing $I$ and consider the ring $S_{\pp}/IS_{\pp}$. Notice that there exists $i$ such that $y_i\notin \pp$, so
\[IS_{\pp}=(x_1,\ldots ,x_n)^2+(x_i+1/y_i\sum_{j\neq i}x_jy_j). \ \]
Therefore, by denoting $S'=\kk[x_j,y_k:j\neq i]$ and $\pp'=\pp\cap S'$, we get
\[\frac{S_{\pp}}{IS_{\pp}}\cong \frac{S'_{\pp'}}{(x_j :j\neq i)^2},\] that is certainly Cohen-Macaulay.

For the last point, notice that we have a short exact sequence:
$$
0\to \frac{S}{(x_1,\dots,x_n)}(-2) \xrightarrow{\cdot (x_1y_1+\ldots +x_ny_n)} \frac{S}{(x_1,\ldots ,x_n)^2} \to \frac{S}{I}\to 0
$$
which induces
$$
0\to H_{\mathfrak{m}}^{n-1}(S/I)\to H_{\mathfrak{m}}^{n}\left(\frac{S}{(x_1,\dots,x_n)}(-2)\right) \xrightarrow{\cdot (x_1y_1+\ldots +x_ny_n)} H_{\mathfrak{m}}^{n}\left(\frac{S}{(x_1,\ldots ,x_n)^2}\right)
$$
So $H_{\mathfrak{m}}^{n-1}(S/I)$ is the kernel, but it is easy to compute that, up to scalar, the only element in the kernel is the element $\frac{1}{y_1\cdots y_n}$, which has degree $2-n$ in $H_{\mathfrak{m}}^{n}(\frac{S}{(x_1,\dots,x_n)}(-2))$. Thus we are done.
\end{remark}

\section{Regularity bounds in characteristic 0}\label{main}

In this section we establish the key vanishing result in the characteristic $0$ that is needed in Theorem \ref{main0} stated in the introduction. As explained in Section \ref{prelim}, we need to establish condition $(MT_r)$ (Definition \ref{defMT}) for nice singularities satisfying Serre's condition $(S_r)$. The following gives a bound on regularity of $\Ext_S^{i}(R, \omega_S)$ which is crucial. We begin by recalling the definition of Du Bois singularities.

Suppose that $X$ is a reduced scheme essentially of finite type over a field of characteristic $0$.  Associated to $X$ is an object $\underline{\Omega}_X^0 \in D^b(X)$ with a map $O_X \to \underline{\Omega}_X^0$. Following \cite[2.1]{Sch}, if $X\subseteq Y$ is an embedding with $Y$ smooth (which is always the case if $X$ is affine or projective), then $\underline{\Omega}_X^0 \cong R\pi_*O_E$ where $E$ is the reduced pre-image of $X$ in a strong log resolution of the pair $(Y, X)$.\footnote{This means $\pi$ is a log resolution of $(Y, X)$ that is an isomorphism outside of $X$. We note that this is not the original definition of $\underline{\Omega}_X^0$ but it is equivalent by the main result of \cite{Sch}.} $X$ is Du Bois if $O_X\to \underline{\Omega}_X^0$ is an isomorphism.

\begin{proposition}\label{p:koda}
If $X = \Proj R$ is Du Bois, then  $$H_\m^j(\Ext_S^{n-i}(R, \omega_S))_{>0}=0$$ for all $i, j$. In particular,
\[\reg \Ext_S^{i}(R, \omega_S)\leq \dim \Ext_S^{i}(R, \omega_S) \ \ \ \forall \ i.\]
\end{proposition}

To this purpose, we will use $\omega_R^\bullet$ to denote the normalized dualizing complex of $R$, thus $h^{-i}(\omega_R^\bullet)\cong \Ext_S^{n-i}(R, \omega_S)$. We use $\underline{\omega}_R^\bullet$ to denote $R\Hom_R(\underline{\Omega}_R^0, \omega_R^\bullet)$. Similarly we have $\omega_X^\bullet$, $\underline{\omega}_X^\bullet$ for $X$. Note that when $R$ (resp. $X$) is Du Bois, we have $\underline{\omega}_R^\bullet={\omega}_R^\bullet$ (resp. $\underline{\omega}_X^\bullet=\omega_X^\bullet$). We begin by recalling some vanishing/injectivity result that we will need:

\begin{theorem}[Theorem 3.3 in \cite{KS}]
\label{theorem--key injectivity}
$h^{-i}(\underline{\omega}_R^\bullet)\hookrightarrow h^{-i}({\omega}_R^\bullet)$ is injective for all $i$.
\end{theorem}

\begin{theorem}[Lemma 3.3 in \cite{KSS} and Theorem 3.2 in \cite{Am}]
\label{theorem--Ambro's vanishing}
$h^{-i}(\underline{\omega}_X^\bullet)$ satisfies the Kodaira vanishing: $H^j(X, h^{-i}(\underline{\omega}_X^\bullet)\otimes L^m)=0$ for all $i$, $j\geq 1$, $m\geq 1$ and $L$ ample line bundle on $X$.
\end{theorem}

\begin{theorem}[Proposition 4.4 and Theorem 4.5 in \cite{MSS}]
\label{theorem--MSS injectivity}
Suppose $X$ is Du Bois, then:
\begin{enumerate}
\item $h^j(\underline{\Omega}_R^0)\cong H_\m^{j+1}(R)_{>0}$ for all $j\geq 1$, and $h^0(\underline{\Omega}_R^0)/R\cong H_\m^1(R)_{>0}$.
\item We have a degree-preserving injection $\Ext_S^{n-i}(R, \omega_S)_{\geq 0} \hookrightarrow H_I^{n-i}(\omega_S)_{\geq 0}$ for all $i$.
\end{enumerate}
\end{theorem}

\begin{theorem}[Theorem 1.2 in \cite{MZ}]
\label{theorem--MZ Eulerian}
Let $S$ be the standard graded polynomial ring in $n$ variables over $\kk$ and let $I$ be a homogeneous ideal. Then all socle elements of $H_\m^0H_I^{j}(S)$ have degree $-n$. In particular, $H_\m^0H_I^{j}(S)_{>-n}=0$ for all $j$.
\end{theorem}

We are now ready to prove Proposition \ref{p:koda}:

\begin{proof}
First we note that $h^{-i}(\omega_R^\bullet)\cong \Ext_S^{n-i}(R, \omega_S)$. When $j\geq 2$, we have
$$H_\m^j(h^{-i}({\omega}_R^\bullet))_{m}=H^{j-1}(X, \widetilde{h^{-i}({\omega}_R^\bullet)}\otimes L^m)=H^{j-1}(X, h^{-i+1}(\omega_X^\bullet)\otimes L^m)=0$$
for all $m>0$ by Theorem \ref{theorem--Ambro's vanishing} (here $L=O_X(1)$ is very ample, and since $X$ has Du Bois singularities $\omega_X^\bullet=\underline{\omega}_X^\bullet$). Thus $H_\m^j(h^{-i}({\omega}_R^\bullet))_{>0}=0$ for all $i$ and $j\geq 2$.

We next consider the case $j=0$. We note that by Theorem \ref{theorem--MSS injectivity} we have $$H_\m^0(\Ext_S^{n-i}(R, \omega_S))_{>0}\hookrightarrow\Ext_S^{n-i}(R, \omega_S)_{>0}\hookrightarrow H_I^{n-i}(\omega_S)_{> 0}\hookrightarrow H_I^{n-i}(\omega_S).$$
Because $H_\m^0(\Ext_S^{n-i}(R, \omega_S))_{>0}$ is clearly $\m$-torsion, the above injection factors through $H_\m^0H_I^{n-i}(\omega_S)$. Thus we have an induced injection $$H_\m^0(\Ext_S^{n-i}(R, \omega_S))_{>0}\hookrightarrow H_\m^0H_I^{n-i}(\omega_S)_{>0}.$$
Now we note that $H_\m^0H_I^{n-i}(\omega_S)_{>0}\cong H_\m^0H_I^{n-i}(S)_{>-n}$. But by Theorem~\ref{theorem--MZ Eulerian}, we have $H_\m^0H_I^{n-i}(S)_{>-n}=0$ and thus the injection above implies $$H_\m^0(h^{-i}({\omega}_R^\bullet))_{>0}=H_\m^0(\Ext_S^{n-i}(R, \omega_S))_{>0}=0$$ for all $i$.

It remains to prove the case $j=1$. By Theorem \ref{theorem--key injectivity} we have
\begin{equation}
0\to h^{-i}(\underline{\omega}_R^\bullet)\to h^{-i}({\omega}_R^\bullet) \to C\to 0
\end{equation}
such that $\Supp C=\{\m\}$ because $X$ is Du Bois. It follows from this sequence that $H_\m^0(h^{-i}(\underline{\omega}_R^\bullet))_{>0}=0$ for all $i$ since it injects into $H_\m^0(h^{-i}({\omega}_R^\bullet))_{>0}=0$.

At this point, we note that the long exact sequence of local cohomology of (1) implies that $H_\m^j(h^{-i}(\underline{\omega}_R^\bullet))\cong H_\m^j(h^{-i}({\omega}_R^\bullet))=0$ for all $j\geq 2$ and that $H_\m^1(h^{-i}(\underline{\omega}_R^\bullet))\twoheadrightarrow H_\m^1(h^{-i}({\omega}_R^\bullet))$. Thus it suffices to show $H_\m^1(h^{-i}(\underline{\omega}_R^\bullet))_{>0}=0$. Since $\omega_R^\bullet$ is the normalized dualizing complex, by the definition of $\underline{\omega}_R^\bullet$, we have
$$\underline{\Omega}_R^0\cong R\Hom_R(\underline{\omega}_R^\bullet, \omega_R^\bullet).$$ Thus we have a spectral sequence $$E_2^{ij}=h^{-j}(R\Hom_R(h^{-i}(\underline{\omega}_R^\bullet), \omega_R^\bullet))\Rightarrow h^{i-j}(\underline{\Omega}_R^0).$$
By graded local duality, $(E_2^{ij})^\vee\cong H_\m^j(h^{-i}(\underline{\omega}_R^\bullet))$, and we already know the latter one vanishes in degree $>0$ for all $i$ and $j\neq 1$, thus $(E_2^{ij})_{<0}=0$ for all $i$ and $j\neq 1$. Therefore the spectral sequence in degree $<0$ degenerates at $E_2$-page. It follows that $$h^{-1}(R\Hom_R(h^{-i}(\underline{\omega}_R^\bullet), \omega_R^\bullet))_{<0}=(E_2^{i1})_{<0}\cong h^{i-1}(\underline{\Omega}_R^0)_{<0} .$$ By Theorem \ref{theorem--MSS injectivity}, $h^{i-1}(\underline{\Omega}_R^0)_{<0}=0$ for all $i$. This implies $$h^{-1}(R\Hom_R(h^{-i}(\underline{\omega}_R^\bullet), \omega_R^\bullet))_{<0}=0$$ and thus by graded local duality $H_\m^1(h^{-i}(\underline{\omega}_R^\bullet))_{>0}=0$.
\end{proof}

\section{Regularity bounds in positive characteristic}\label{mainp}

The purpose of this section is to prove  the analog of Proposition \ref{p:koda} in characteristic $p>0$. The natural analogous assumption would be to assume $X=\Proj R$ is $F$-injective. Albeit we do not know any counterexample to the thesis of Theorem \ref{main0} under the assumption that $X$ is $F$-injective, we know that Proposition \ref{p:koda} is not true in positive characteristic (as Kodaira vanishing fails even if $X$ is smooth). As usual for this business, we can settle the problem  assuming that $R$ itself is $F$-pure.

\begin{proposition}\label{sp}
If $\chara(\kk)=p>0$ and $R$ is $F$-pure, then
\[H_{\mm}^j(\Ext_S^i(R,\omega_S))_{>0}=0 \ \ \ \forall \ i,j.\]
In particular, $\reg(\Ext_S^{i}(R,\omega_S))\leq \dim \Ext_S^{i}(R,\omega_S)$ for all $i$.
\end{proposition}
\begin{proof}
Let $q=p^t$ for some $t\in\N$. Given an $R$-module $M$, we denote by $^q \!M$ the $R$-module that is $M$ as an abelian group and has $R$-action given by $r\cdot m=r^{q}m$. The statement allows to enlarge the field, thus we can assume that $\kk$ is perfect.
Under the assumptions, by \cite[Corollary 5.3]{HR} the $t$th iterated Frobenius power $F^t:R \ \rightarrow \ ^q \!R$ \ splits as a map of $R$-modules, so for any $j=0,\ldots ,n$ the induced map on local cohomology
\[H_{\mm}^j(R)\xrightarrow{H_{\mm}^j(F^t)} H_{\mm}^j(^q \! R)\]
is an injective splitting of $R$-modules as well. Notice that
a degree $s$ element in $H_{\mm}^j(R)$ is mapped by $H_{\mm}^j(F^t)$ to a degree $sq$ element of $H_{\mm}^j(^q \! R)$.
By abuse of notation, from now on we will write $F^t$ for $H_{\mm}^j(F^t)$. Let us apply to such map the functor $-^{\vee}=\Hom_{\kk}(-,\kk)$, obtaining the surjective splitting of $R$-modules:
\[H_{\mm}^j(R)^{\vee}\xleftarrow{(F^t)^{\vee}} H_{\mm}^j(^q \!R)^{\vee}.\]
Furthermore, notice that if $\eta^*$ is a degree $u$ element in $H_{\mm}^j(^q \!R)^{\vee}$, then $(F^t)^{\vee}(\eta^*)=\eta^*\circ F^t$ is 0 if $q$ does not divide $u$, while it has degree $s$ if $u=qs$. By graded Grothendieck's duality, we therefore have a surjective splitting of $R$-modules
\[\Ext_S^{n-j}(R,\omega_S)\xleftarrow{(F^t)^{\vee}}\Ext_S^{n-j}(^q \!R,\omega_S)\cong \ ^q \!\Ext_S^{n-j}(R,\omega_S)\]
which ``divides" the degrees by $q$. By applying the local cohomology functor $H_{\mm}^k(-)$ to the above splitting, we get a surjective map of graded $R$-modules
\[H_{\mm}^k(\Ext_S^{n-j}(R,\omega_S))\xleftarrow{H_{\mm}^k((F^t)^{\vee})}H_{\mm}^k(^q\Ext_S^{n-j}(R,\omega_S))\cong \ ^q \!H_{\mm}^k(\Ext_S^{n-j}(R,\omega_S))\]
dividing the degrees by $q$.
Hence, because $q$ can be chosen arbitrarily high, the graded surjection of $R$-modules $H_{\mm}^k((F^t)^{\vee})$ yields that $H_{\mm}^k((\Ext_S^{n-j}(R,\omega_S)))_s$ is actually 0 whenever $s>0$.
\end{proof}

\section{Applications and examples}\label{apps}
In this section, we first put together our technical results to prove main Theorem \ref{main0}.
\begin{proof}[Proof of Theorem \ref{main0}]
Assume $R$ satisfies $(S_r)$ and one of the assumptions: $X=\Proj R$ is Du Bois (characteristic $0$) or $R$ is $F$-pure (characteristic $p$). We can extend $\kk$ if necessary and hence assume it is algebraically closed. By Theorem \ref{mtr3} and Proposition \ref{mtr4}, all the conclusions  follow if we can establish that $R$ is $(MT_r)$. But  Propositions \ref{algSl} tells us that $\dim  \Ext_S^{n-i}(R,\omega_S) \leq i-r $, and Propositions \ref{p:koda} and \ref{sp} assert that  $\reg  \Ext_S^{n-i}(R,\omega_S) \leq \dim \Ext_S^{n-i}(R, \omega_S)$, which gives precisely what we want.
\end{proof}

The assumptions of the above theorem can be slightly relaxed: for example if the ideal $I$ admits a squarefree initial ideal (e.g. if $I$ is a binomial edge ideal or $R$ an algebra with straightening law) then, even if $R$ might not be $F$-pure (see \cite[Corollary 5.1, Remark 5.2]{KV} for a discussion on this), the conclusion of the theorem continues to hold. To formalize this let us introduce the following concept:

\begin{definition}
We say that an $\N$-graded $R_0$-algebra $R$ with $R_0$ a field of positive characteristic, is {\it deformation equivalent to an $F$-pure ring} (de$F$pure for short) if there exist a domain $A$, essentially of finite type over $\kk$, and a flat finitely generated positively graded $A$-algebra $R_A=\oplus_{i\in\N}[R_A]_i$ with ($[R_A]_0=A$), one fibre is $R$ and one is an $F$-pure standard graded algebra. In other words, there exist prime ideals $\pp,\qq \in A$ such that $R_A\otimes_A\kappa(\pp)\cong R$ and $R_A\otimes_A\kappa(\qq)$ is an $F$-pure standard graded $\kappa(\qq)$-algebra.
\end{definition}

Obviously $R$ is de$F$-pure whenever it is $F$-pure (take $A=\kk$ and $\pp=\qq=\{0\}$). More interestingly, any graded Algebra with Straightening Law is de$F$-pure. On the contrary, de$F$-pure rings need not be $F$-injective: e.g. $R=\kk[x,y,z]/(x^3+y^3+z^3+xyz)$ is de$F$-pure but not $F$-injective if $\kk$ has characteristic $5$ (take $A=\kk[t]$, $R_A=\kk[x,y,z,t]/(tx^3+ty^3+tz^3+xyz)$, $\pp=(t-1)$ and $\qq=(t)$). 

\begin{corollary}
Let $R$ be a $\N$-graded algebra satisfying $(S_r)$ over a field of positive characteristic $R_0$. If $R$ is de$F$-pure the same conclusion of Theorem \ref{main0} holds.
%
%
\end{corollary}
\begin{proof}
If $R_A\otimes_A\kappa(\pp)$ satisfies $(S_r)$, then the general fibre $R_A\otimes_A\kappa(0)$ satisfies $(S_r)$ as well by \cite[Theorem 4.5]{AF}. On this context, if $r\geq 2$, $R_A\otimes_A\kappa(0)$ satisfies $(S_r)$ if and only if $\dim(\Ext^{n-i}_{S_A\otimes_A\kappa(0)}(R_A\otimes_A\kappa(0),S_A\otimes_A\kappa(0)))$ is at most $i-r$ for all $i<r$, (see Proposition \ref{algSl}). By \cite[Corollary 1.5]{KK} the same dimension inequality holds true for the fibre at $\qq$, since $R_A\otimes_A\kappa(\qq)$ is $F$-pure. So $\dim(\Ext^{n-i}_{S_A\otimes_A\kappa(\qq)}(R_A\otimes_A\kappa(\qq),S_A\otimes_A\kappa(\qq)))$ is at most $i-r$ for all $i<r$, hence $R_A\otimes_A\kappa(\qq)$ satisfies $(S_r)$ condition by Proposition \ref{algSl}. At this point, since $R$ and $R_A\otimes_A\kappa(\qq)$, being fibres of the same flat family, have the same Hilbert function, we can replace the latter to the first one and exploit Theorem \ref{main0}.
\end{proof}

The $h$-vector inequalities can be translated to information about the Hilbert functions. Below we give some precise statements. Recall that $R=S/I$, $e= \height I, d=\dim R$. Let $s_i =  \dim_k I_i$.
\begin{proposition}
For each $l\geq 0$, $$h_l = {e+l-1 \choose l} - \sum_{j=0}^l (-1)^js_{l-j}{d \choose j}.$$
\end{proposition}

\begin{proof}
The Hilbert series of $R$ can be written as:
$$H_R(t)= \sum r_it^i = \frac{h_0+h_1t+ \cdots + h_lt^l}{(1-t)^d}  \cdots (*) $$
where $r_i=\dim_{\kk}R_i$. By comparing coefficients in $(*)$ we always have:
$$h_l =  \sum_{j=0}^l (-1)^jr_{l-j}{d \choose j} $$
Since $s_i= {n+i-1 \choose i} - r_i$, what we need to prove amount to:
$$\sum_{j=0}^l (-1)^j{n+l-j-1 \choose l-j}{d \choose j} =   {e+l-1 \choose l}$$
But we have $$\sum_{i\geq 0} {n+i-1 \choose i}t^i = \frac{1}{(1-t)^n}$$
Since $n=d+e$, we have
$$\left(\sum_{i\geq 0} {n+i-1 \choose i}t^i\right)(1-t)^d = \frac{1}{(1-t)^e}$$
The $l$-th coefficient on both sides give exactly the equality we seek.
\end{proof}

From the equality above, the following is an immediate consequence of our main Theorem \ref{main0}. It says, for instance, that if a homogenous ideal $I$ in $R=\kk[x_1,\dots, x_n]$ is such that $R/I$ is $(S_3)$ and $\Proj(R/I)$ is Du Bois, and $I$ contains no linear or quadratic forms, then the number of cubic generators of $I$ is at most $\binom{\height I+2}{3}$. Note that such result is not obvious even in the special case $\height I=2$, and $R/I$ is a Cohen-Macaulay (or just depth $3$ or more isolated singularity). One can view such statements as generalizations of the classical bound on number of quadrics mentioned in the introduction.

\begin{corollary}\label{cor1}
Let $R=\kk[x_1,\dots, x_n]/I$ be a standard graded algebra over a field $\kk$ with $d=\dim R$ and $e=\height I$. Let $s_i= \dim_k I_i$. Assume $R$ satisfies $(S_r)$ and either:
\begin{enumerate}

\item $\chara(\kk)=0$ and $X=\Proj R$ is Du Bois.
\item $\chara(\kk)=p>0$ and $R$ is $F$-pure (or deformation equivalent to an $F$-pure ring).
\end{enumerate}
Then for each $l\leq r$, we have:
$$\sum_{j=0}^l (-1)^js_{l-j}{d \choose j}\leq {e+l-1 \choose l}.$$
In particular, if $I$ contains no elements of degree less than $l$, then $s_l \leq {e+l-1 \choose l}$. Moreover, if equality happens for any such $l$, then $R$ is Cohen-Macaulay.
\end{corollary}

\begin{remark}
In \cite{KM}, it is established that for a monomial ideal $I$ (not necessarily square-free), the deficiency modules of $R=S/I$ satisfies the same regularity bounds as in Propositions \ref{p:koda} and \ref{sp}, so our main results above applied for this situation as well.
It would be interesting to know if the conclusion of Theorem \ref{main0} holds replacing $R$ with $R_{\mathrm{red}}=S/\sqrt{I}$ in the assumptions (1) and (2): we note that, if $R_{\mathrm{red}}$ is a Du Bois singularity (which is more restrictive than asking that $\Proj R$ is Du Bois) in characteristic 0 or $R$ is $F$-pure in positive characteristic, then, using Proposition \ref{algSl} $R_{\mathrm{red}}$ satisfies $(S_r)$ whenever $R$ does by \cite[Remark 2.4 (2) and (3)]{DDS}. In the next proposition we point out a case where we can apply recent vanishing theorems obtained in \cite{BBLSZ}.
 \end{remark}

 \begin{proposition}
 Let $J\subset S=\kk[x_1,\ldots ,x_n]$ be a homogeneous ideal such that $\Proj S/J$ is smooth, and assume that $\kk$ is a field of characteristic 0. Given a positive integer $t$, consider the saturation $I$ of $J^t$ and $(h_0,\ldots ,h_s)$ the $h$-vector of $R=S/I$. If $R$ satisfies $(S_r)$, then $h_i\geq 0$ for $i=0,\ldots , r$. Also, $h_r+ h_{r+1}+\dots+h_s \geq 0$, or equivalently $R$ has multiplicity at least $h_0+h_1+\dots +h_{r-1}$.
If furthermore $R$ has Castelnuovo-Mumford regularity less than $r$ or $h_i=0$ for some $i\leq r$,  then $R$ is Cohen-Macaulay.
 \end{proposition}
 \begin{proof}
Calling $X=\Proj S/I$, by \cite[Theorem 1.4]{BBLSZ} we have $H^k(X,\O_X(-j))=0$ for all $k<\dim X$ and $j>0$. In other words, $H^k_{\mm}(R)_{<0}=0$ for all $k=2,\ldots ,\dim R-1$. The conclusion is obvious if $r\leq 1$,  so we can assume $r\geq 2$. Hence $H^0_{\mm}(R)=H^1_{\mm}(R)=0$. By graded duality, we have $\Ext^{n-i}_S(R,\omega_S)_{>0}=0$ for any $i<\dim R$. Therefore $R$ satisfies $(MT_r)$ by Proposition \ref{algSl}, and we conclude by Theorem \ref{mtr3} and Proposition \ref{mtr4}.
 \end{proof}

The next corollary fits the well-known but mysterious theme that ``nice singularities of small codimension" should be Cohen-Macaulay.

\begin{corollary}\label{cor2}
Let $R=\kk[x_1,\dots, x_n]/I$ be a standard graded algebra over a field $\kk$ of characteristic $0$ with $e=\height I$ and $d=\dim R$. Let $d_1 \geq d_2 \geq \dots $ be the degree sequence of a minimal set of generators for $I$. Assume that $R$ is unmixed, equidimensional, Cohen-Macaulay in codimension $l$ and $X=\Proj R$ has only MJ-log canonical singularities.\footnote{We refer to \cite{EI} for detailed definition and properties of MJ-log canonical singularities.} If $e+l \geq d_1+\dots +d_e$, then $R$ is Cohen-Macaulay.
\end{corollary}
\begin{proof}
By \cite[Proposition 2.7]{Niu} (see also \cite{Ish}), $X$ is MJ-log canonical is equivalent to saying that the pair $(\mathbb{P}^{n-1}, \widetilde{I}^e)$ is log canonical. Let $r= d_1+\dots + d_e$. Now applying \cite[Corollary 5.1]{dFE}, we have
$$H_\m^i(I)_j=0 \text{ for all $i>1$ and $j\geq r-n+1$. }$$
Since we have $0\to I\to S\to R\to 0$, the long exact sequence of local cohomology tells us that $H_\m^i(R)_j=0$ for $0<i<d$ and $j\geq r-n+1$. This is  equivalent to saying that  for each such $i$, ${[K_i]}_{\leq n-r-1}=0$ with $K_i = \Ext_S^{n-i}(R,\omega_S)$. However, since MJ-log canonical singularities are Du Bois by \cite[Theorem 7.7]{dFD}, Proposition \ref{p:koda} implies that $\reg K_i\leq \dim K_i \leq d-l-1$. Our assumption says that $d-l-1\leq  n-r-1$, thus $K_i=0$ for all $i<d$ and hence $R$ is Cohen-Macaulay.
\end{proof}

\begin{corollary}\label{cor3}
Let $R=\kk[x_1,\dots, x_n]/I$ be a standard graded algebra over a field $\kk$ of characteristic $0$. Let $e=\height I$ and $d_1 \geq d_2 \geq \dots $ be the degree sequence of a minimal set of generators for $I$. Assume that $\Proj R$ is a  smooth variety, $n\geq 2e+3$ and $n> d_1+\dots +d_e$. Then $R$ is Gorenstein. In particular, if $e=2$ then $R$ is a complete intersection.
\end{corollary}

\begin{proof}
By Corollary \ref{cor2}, $R$ is Cohen-Macaulay. Let $X=\Proj R$. The assumption of smoothness and $n\geq 2e+3$ implies that $\Pic(X) = \Z$, generated by $\mathcal O_X(1)$, see \cite[Theorem 11.4]{Lyu}.  It follows that the class group of $R$ is trivial, and therefore the canonical module is free.
\end{proof}

Our methods also give streamlined proof and sometimes strengthen known results. Here is a statement on $h_2$ that can be seen as a modest extension of \cite[Theorem 4.2]{EG}.

\begin{theorem}\label{h2}
Let $R=S/I=\kk[x_1,\dots, x_n]/I$ be a standard graded algebra over an algebraically closed field $\kk$ with $e=\height I$. Suppose that $\Proj R$ is connected in codimension one and the radical of $I$ contains no linear forms. Then $h_2(R)\geq 0$ and $e(R)\geq 1+ e$. If equality happens in either case, then $R_{\text{red}}= S/\sqrt{I}$ is Cohen-Macaulay of minimal multiplicity.\footnote{This means $R$ is Cohen-Macaulay and $e(R)=e+1$ in our context.}
\end{theorem}

\begin{proof}
Let $J= \sqrt{I}$ and $T=R_{\text{red}}=S/J$. By assumption, $J, I$ contains no linear forms, so $h_2(R) ={e+1 \choose 2}-\dim_{\kk}I_2$ and $h_2(T) ={e+1 \choose 2}-\dim_{\kk}J_2$. It follows that $h_2(R)\geq h_2(T)$. One also easily see  that $e(R) \geq e(T)$. So we may replace $R$ by $T$ and assume $R$ is reduced. We may assume now that $d =\dim R\geq 2$.

Let $l$ be a generic linear form, $\bar R= R/lR$ and $R'={\bar R}/N$ with $N=H^0_{\m}(\bar R)$.
Bertini Theorem tells us that $\Proj R'$ is still connected in codimension one and reduced. Since $R$ is standard graded, we have short exact sequence of graded $R$-modules
$$0\to H_\m^0(R)\to R\to \oplus_{n\in\mathbb{Z}}H^0(X, O_X(n)) \to H_\m^1(R)\to 0$$
where $X=\Proj R$. Now clearly $H^0(X, O_X(n))=0$ when $n<0$ and $H^0(X, O_X)=\kk$ by our assumption that $X$ is connected and $\kk$ is algebraically closed. The exact sequence above then tells us that
$H^1_{\m}(R)_{\leq 0}=0$, so $H_{\m}^0(\bar R)_{\leq 1}=0$. It follows as in the proof of Proposition \ref{mtr2} that $h_i(R')=h_i(R)$ for $i<2$, $h_2(R')\leq h_2(R)$ and $c_2(R)=c_2(R')$. Thus by induction on dimension, we can assume $\dim R'= \dim \bar R=2$. But here $H_{\mm}^0(\bar R)_{\leq 1}= H_{\mm}^1(\bar R)_{\leq 0} =0$ say precisely that $\bar R$ is $(MT_2)$ (see Proposition \ref{easy}). Our Theorem \ref{mtr3} now applies to show the needed inequalities (note that $e+1= h_0(R) + h_1(R)$).  If any equality happens, then Theorem \ref{mtr4} asserts that $\bar R$ is Cohen-Macaulay (necessarily of regularity $2$). To finish we now claim that if $\dim R\geq 2$ and $R/lR$ is Cohen-Macaulay for a generic $l$, then so is $R$. Let $N =H_\m^0(R)$ and $R' = R/N$. We obtain an exact sequence $0\to N/lN \to R/lR \to R'/lR' \to 0$. But since $R/lR$ is Cohen-Macaulay of dimension at least $1$, $N/lN=0$, which forces $N=0$, thus $l$ is regular on $R$ and we are done.
\end{proof}

\begin{eg} It is important that $\kk$ is algebraically closed in the previous theorem. Take $R = \R[s,t,is,it]$ which is a domain. Then $R \cong S/I$ with $S =\R[a,b,c,d]/I$ and $I= (a^2+c^2, b^2+d^2, ad-bc, ab+cd)$ and $h_2(R)=-1$. Note that $R\otimes_{\R} \C  \cong \C[x,y,u,v]/(xu,xv,yu,yv)$.
\end{eg}
Next, we give a formula relating the $h$-vectors of $R$ and the (deficiency) $\Ext$ modules. Recall the notations at the beginning of section \ref{prelim}.

\begin{proposition}\label{abs}
Let $R=S/I$ be a standard graded algebra of dimension $d$. For $i=0,\ldots ,d$ set $K_i=\Ext_S^{n-i}(R,\omega_S)$. Then we have the following relations between the numerators of Hilbert series:
$$p_R(1/t)t^d = \sum_{i=0}^d(-1)^{d-i}p_{K_i}(t)(1-t)^{d-\dim K_i} $$

\end{proposition}

\begin{proof}
Applying the main Theorem of \cite{ABS} with $M=R, N=\omega_S = S(-n)$, we obtain an equality of rational functions:
$$\sum_i (-1)^{n-i} H_{K_i}(t) = \frac{H_R(1/t)H_{\omega_S}(t)}{H_S(1/t)} $$
Since $H_{K_i}(t)= \frac{p_{K_i}(t)}{(1-t)^{\dim K_i}}, H_{\omega_S}(t) = \frac{t^n}{(1-t)^n}, H_S(1/t) = \frac{t^n}{(t-1)^n}$, the assertion follows.
\end{proof}

\begin{corollary}
\label{cor4}
Retain the notations of Proposition \ref{abs}. Suppose that for $i<d$, we have $H_{m}^i(R)_{>0} =0$ (for example, this holds if $R$ is $F$-pure or if $R$ is Du Bois \cite[Theorem 4.4]{Ma}). Then  $$h_d(R) = \sum_{0\leq i \leq d}(-1)^{d-i}\dim_{\kk} H^i_{\m}(R)_0.$$
\end{corollary}

\begin{proof}
By (graded) local duality, our assumption implies that for each $i<d$, $[K_i]_{<0}=0$, so the polynomial $p_{K_i}(t)$ contains no negative powers of $t$. What we need to prove follows immediately from Proposition \ref{abs}.
\end{proof}

Next we point out that in certain situations, the negativity of  $h_d(R)$ is guaranteed if $R$ is not Cohen-Macaulay. This class of examples shows that one can not hope to improve of our main theorem too  much.

\begin{corollary}\label{corh_d}
Let $R=\kk[x_1,\dots, x_n]/I$ be a standard graded algebra with $\dim R=d$. Suppose that $\depth R=d-1 = \reg (R)=d-1$ and $H_{\m}^d(R)_{-1}=0$. Then $h_d(R)<0$.
\end{corollary}
\begin{proof}
Since $\depth R=d-1$, $\reg R=d-1$ is equivalent to the condition $H_{\m}^{d-1}(R)_{>0}=H_{\m}^d(R)_{\geq 0} =0$. Since $H_\m^d(R)_{-1}=0$ and $\reg R=d-1$, we have $H^{d-1}_{\m}(R)_0\neq 0$. Now by Corollary \ref{cor4}, $h_d(R) = - \dim_{\kk} H^{d-1}_{\m}(R)_0<0$.
\end{proof}

\begin{eg}
Let $A$ be a standard graded Cohen-Macaulay  $\kk$-algebra of dimension at least $2$ with zero $a$-invariant and $B=\kk[x,y]$. Then the Segre product $R= A\sharp B$ satisfies the assumptions of Corollary \ref{corh_d}. It is even $(S_{d-1})$.
\end{eg}
\section{Some further results and open questions}\label{speculate}

Our results and examples suggest the following natural question, which can be viewed as a vast generalization of everything discussed so far:

\begin{question}\label{openques}
Assume that $\chara \kk=0$ and $R$ satisfies $(S_r)$ and is Du Bois in codimension $r-2$. If $(h_0,\ldots ,h_s)$ is the  $h$-vector of $R$, is it true that $h_i\geq 0$ whenever $i=0,\ldots ,r$? If furthermore $R$ has Castelnuovo-Mumford regularity less than $r$, is it true that $R$ is Cohen-Macaulay?
\end{question}

\begin{remark}
Question \ref{openques} is true if $r=2$: in fact we can assume $\kk=\bar \kk$ and that $I$ contains no linear forms.
The condition that $R$ is Du Bois in codimension 0, together with $(S_1)$, says that $I=\sqrt{I}$. Further, the condition that $R$ satisfies $(S_2)$ yields that $\Proj R$ is connected in codimension 1 by \cite[Proposition 2.1, Theorem 2.2]{Ha}. So we can apply Theorem \ref{h2}. 
\end{remark}

The first step to answer Question \ref{openques} would be to relax the condition $X =\Proj R$ is Du Bois  (i.e., $R$ is Du Bois in codimension $\leq d-1$) in key Proposition \ref{p:koda} to $R$ being Du Bois in lower codimension. This amounts to strong Kodaira-vanishing type statements.   At this point we can prove the following, which is a generalization of \cite[Theorem 3.15]{HuSm}.

\begin{proposition}
Let $0<k\leq d-1$. Suppose $X$ is $(S_{d-k})$ and is Du Bois in codimension $d-k-1$. Then we have $H_\m^{j}(R)_{<0}=0$ for all $j\leq d-k$, or equivalently, $\Ext_S^{n-j}(R,\omega_S)_{>0}=0$ for all $j\leq d-k$.
\end{proposition}
\begin{proof}
If $k=d-1$, we know $X$ is $(S_1)$ and generically reduced and thus $X$ is reduced. Then $H_\m^1(R)_{<0}$ vanishes. For the case $k<d-1$, notice that it is enough to show $H_\m^{d-k}(R)_{<0}=0$ (because the assumption remains valid for any $k'\geq k$). Hence it suffices to prove $H^{d-k-1}(X, L^{-m})=0$ for all $m\geq 1$ (here we used $k<d-1$). Since $X$ is $(S_{d-k})$, the possible non-vanishing cohomology of $\omega_X^\bullet$ are:
$$h^{-(d-1)}(\omega_X^\bullet), h^{-(d-1)+1}(\omega_X^\bullet),\dots, h^{-(d-1)+(k-1)}(\omega_X^\bullet),$$ and we have $$\dim h^{-(d-1)+1}(\omega_X^\bullet)\leq k-2, \dim h^{-(d-1)+2}(\omega_X^\bullet)\leq k-3, \dots, \dim h^{-(d-1)+(k-1)}(\omega_X^\bullet)\leq 0.$$
By Serre duality, $H^{d-k-1}(X, L^{-m})$ is dual to $\mathbb{H}^{-d+k+1}(X, \omega_X^\bullet\otimes L^m)$. Using spectral sequence, the possible $E_2$-page contributions to $\mathbb{H}^{-d+k+1}(X, \omega_X^\bullet\otimes L^m)$ are: $$H^k(X, h^{-(d-1)}(\omega_X^\bullet)\otimes L^m), H^{k-1}(X, h^{-(d-1)+1}(\omega_X^\bullet)\otimes L^m),\dots, H^1(X, h^{-(d-1)+(k-1)}(\omega_X^\bullet)\otimes L^m).$$ The latter ones $$H^{k-1}(X, h^{-(d-1)+1}(\omega_X^\bullet)\otimes L^m),\dots, H^1(X, h^{-(d-1)+(k-1)}(\omega_X^\bullet)\otimes L^m)$$ all vanish because dimension reasons (for example, $H^{k-1}(X, h^{-(d-1)+1}(\omega_X^\bullet)\otimes L^m)=0$ since $\dim h^{-(d-1)+1}(\omega_X^\bullet)\leq k-2$).

For $H^k(X, h^{-(d-1)}(\omega_X^\bullet)\otimes L^m)$, note that by Theorem \ref{theorem--key injectivity} we have a short exact sequence $$0\to h^{-(d-1)}(\underline{\omega}_X^\bullet)\to h^{-(d-1)}(\omega_X^\bullet)\to C\to 0$$ such that $\dim \Supp C\leq k-1$ since $X$ is Du Bois in codimension $(d-1)-k$. The induced long exact sequence of sheaf cohomology: $$0=H^k(X, h^{-(d-1)}(\underline{\omega}_X^\bullet)\otimes L^m)\to H^k(X, h^{-(d-1)}(\omega_X^\bullet)\otimes L^m)\to H^k(X, C\otimes L^m)=0,$$ where the first $0$ is by Theorem \ref{theorem--Ambro's vanishing}, shows that $H^k(X, h^{-(d-1)}(\omega_X^\bullet)\otimes L^m)=0$ for all $m\geq 1$. Hence all $E_2$-page contributions to $\mathbb{H}^{-d+k+1}(X, \omega_X^\bullet\otimes L^m)$ are $0$, it follows that $\mathbb{H}^{-d+k+1}(X, \omega_X^\bullet\otimes L^m)=0$ for all $m\geq 1$ and thus $H^{d-k-1}(X, L^{-m})=0$ for all $m\geq 1$.
\end{proof}

\end{document}